\documentclass[a4paper,12pt]{article}
\usepackage[centertags]{amsmath}
\usepackage{amsfonts}
\usepackage{amssymb}
\usepackage{amsthm}
\usepackage{dsfont}
\usepackage{tikz, subfigure}
\usepackage{verbatim}
\newtheorem{theorem}{Theorem}[section]
\newtheorem{lemma}{Lemma}[section]
\newtheorem{cor}{Corollary}[section]
\newtheorem{prop}{Proposition}[section]
\newtheorem{defn}{Definition}[section]

\newtheorem{example}{Example}

\begin{document}

\title{xxxx}
 \title{Sets of $\beta$-expansions and the Hausdorff Measure of Slices through Fractals}
\author{Tom Kempton}
\maketitle
\begin{abstract}We study natural measures on sets of $\beta$-expansions and on slices through self similar sets. In the setting of $\beta$-expansions, these allow us to better understand the measure of maximal entropy for the random $\beta$-transformation and to reinterpret a result of Lindenstrauss, Peres and Schlag in terms of equidistribution. Each of these applications is relevant to the study of Bernoulli convolutions. In the fractal setting this allows us to understand how to disintegrate Hausdorff measure by slicing, leading to conditions under which almost every slice through a self similar set has positive Hausdorff measure, generalising long known results about almost everywhere values of the Hausdorff dimension.\end{abstract}
\section{Introduction}
Given $\beta\in(1,2)$, a $\beta$-expansion of a real number $x$ is a sequence $\underline a\in\{0,1\}^{\mathbb N}$ for which
\[
\pi_{\beta}(\underline a):=\sum_{i=1}^{\infty}a_i\beta^{-i}=x.
\]
We let $\mathcal E_{\beta}(x):=\pi_{\beta}^{-1}(x)$ denote the set of $\beta$-expansions of $x$. 

The primary purpose of this article is to seek to understand measures on $\mathcal E_{\beta}(x)$. In particular, we study the family of measures $m_x:=m|_{\mathcal E_{\beta}(x)}$ obtained by disintegrating the uniform $(\frac{1}{2},\frac{1}{2})$ Bernoulli measure $m$ on $\{0,1\}^{\mathbb N}$. These measures appear as disintegrations of the measure of maximal entropy for the random $\beta$-transformation in \cite{DdV1}, and are used to state an equidistribution result for $\beta$-expansions in \cite{LPS}.

We begin by assuming that the Bernoulli convolution $\nu_{\beta}$ (defined later) is absolutely continuous. In this setting we build a two-dimensional dynamical system which preserves Lebesgue measure and for which vertical fibres through the state space correspond to the sets $\mathcal E_{\beta}(x)$. By lifting one dimensional Lebesgue measure on these fibres to the sets $\mathcal E_{\beta}(x)$ we obtain formulae for $m_x$ in terms of the density of $\nu_{\beta}$. 



We also consider Hausdorff measure on $\mathcal E_{\beta}(x)$. Results on the cardinality, branching rate and dimension of $\mathcal E_{\beta}(x)$ were given in a series of recent papers \cite{BakerGrowth, FengSidorov, CountingBeta, Sidorov1}. We continue this line of research by showing that for certain $\beta$, including almost all $\beta\in(1,\sqrt{2})$, the set $\mathcal E_{\beta}(x)$ has positive finite Hausdorff measure, and in that case the normalised Hausdorff measure on $\mathcal E_{\beta}(x)$ coincides with $m_x$. Our necessary and sufficient condition for the positivity of Hausdorff measure is that the Bernoulli convolution $\nu_{\beta}$ is absolutely continuous with bounded density.

We then use the formulae for the measures $m_x$ obtained by our natural extension to reinterpret the results of \cite{LPS} as equidistribution results for the sets $\mathcal E_{\beta}(x)$. In particular, we show that for almost all $\beta\in(1,\sqrt{2})$ and almost all $x\in I_{\beta}$ the sets 
\[
\mathcal O^n(x):=\{\pi_{\beta}(\sigma^n(\underline{a})):\underline a\in \mathcal E_{\beta}(x)\} 
\]
equidistribute with respect to Lebesgue measure as $n\to\infty$, where $\sigma$ denotes the left shift. Hochman proved in \cite{HochmanOverlaps} that $\inf\{|x-y|:x,y\in \mathcal O^n(x)\}$ tends to zero superexponentially whenever $\nu_{\beta}$ has dimension less than $1$. We conjecture that the sets $\mathcal O^n(x)$ equidistribute if and only if $\nu_{\beta}$ is absolutely continuous. We are also able to use our results to prove a finer result (Proposition \ref{GrowthProp}) about the typical branching rate of sets of $\beta$-expansions, making progress towards Conjecture 1 of \cite{CountingBeta}.



For each statement that we make about sets of $\beta$-expansions and Bernoulli convolutions, there is an analogous statement about slices through self similar sets and projections of Hausdorff measure. We let $E\subset \mathbb R^n$ be a self similar set of Hausdorff dimension $s$, where the similarities do not include rotations and  satisfy another technical condition (Definition \ref{DeltaDefn}). We let $E_{\theta}$ be the orthogonal projection of $E$ onto the line passing through the origin at angle $\theta=(\theta_1,\theta_2,\cdots,\theta_{n-1})$. We let $E_{\theta,x}$ be the intersection of $E$ with the $(n-1)$-dimensional plane perpendicular to $E_{\theta}$ and passing through $x\in E_{\theta}$. We call the sets $E_{\theta,x}$ slices of $E$.

Our main theorem for fractals, Theorem \ref{FractalThm}, states that $\mathcal H^{s-1}(E_{\theta,x})>0$ for almost every $x\in E_{\theta}$ if and only if the orthogonal projection of Hausdorff measure on $E$ to $E_{\theta}$ is absolutely continuous with bounded density. An example application is the following, we recall that the Menger sponge  is the self similar set defined recursively by subdividing $[0,1]^3$ into 27 subcubes of side length $\frac{1}{3}$, discarding the subcube at the centre of each face of our original cube and the subcube in the centre of our original cube, and then repeating the process for each of the 20 remaining subcubes

\begin{example}
Let $E$ be the Menger sponge. Then almost every plane slice through $E$ has positive finite $\left(\frac{\log(20)}{\log(3)}-1\right)$-dimensional Hausdorff measure.
\end{example}
Corresponding theorems due to Marstrand for the dimension of slices through fractals are well known, but the extension to the case of Hausdorff measure of slices through fractals is new.

In the final section we state a number of open questions related to our work.

\section{Preliminaries}
Let $\Sigma:=\{0,1\}^{\mathbb N}$. We define the left shift $\sigma:\Sigma\to\Sigma$ by \[\sigma(a_1a_2a_3\cdots)=(a_2a_3\cdots).\] Given a word $a_1\cdots a_n\in\{0,1\}^n$ we let the cylinder $[a_1\cdots a_n]$ be given by
\[
[a_1\cdots a_n]:=\{\underline b\in\Sigma:b_1\cdots b_n=a_1\cdots a_n\}.
\]
We let $m$ be the $(\frac{1}{2},\frac{1}{2})$ Bernoulli measure on $\Sigma$, $m$ gives measure $2^{-n}$ to each cylinder $[a_1\cdots a_n]$. 

The Bernoulli convolution $\nu_{\beta}$ is the probability measure on $I_{\beta}:=[0,\frac{1}{\beta-1}]$ defined by
\[
\nu_{\beta}:=m\circ \pi_{\beta}^{-1}.
\]
An alternative definition of $\nu_{\beta}$ is that it is the unique probability measure satisfying the self similarity relation
\[
\nu_{\beta}=\frac{1}{2}\left(\nu_{\beta}\circ T_0 + \nu_{\beta}\circ T_1\right)
\]
where the functions $T_i:\mathbb R\to \mathbb R$ are given by $T_i(x):=\beta x-i$.

There are a number of fascinating open questions relating to Bernoulli convolutions including the fundamental question of for which values of $\beta$ the corresponding Bernoulli convolution is absolutely continuous. Solomyak \cite{SolomyakAC} showed that $\nu_{\beta}$ is absolutely continuous for Lebesgue almost all $\beta\in(1,2)$, and has continuous density for almost all $\beta\in(1,\sqrt{2})$. Mauldin and Simon \cite{MauldinSimon} showed that $\nu_{\beta}$ is actually equivalent to Lebesgue measure whenever it is absolutely continuous. Very recently, Shmerkin \cite{ShmerkinExceptional} has shown that the set of $\beta$ for which $\nu_{\beta}$ is singular has Hausdorff dimension zero.

We let $m_x$ be the disintegration of $m$ by fibres $\mathcal E_{\beta}(x)$. This means that $(m_x)$ is the $\nu_{\beta}$-almost everywhere unique family of measures satisfying that each $m_x$ is a probability measure supported on the fibre $\mathcal E_{\beta}(x)$ and that for every integrable function $f:\Sigma\to\mathbb R$ we have
\begin{equation}\label{disintegration}
\int_{\Sigma} f(\underline a)dm(\underline a)=\int_{I_{\beta}}\int_{\mathcal E_{\beta}(x)}f(\underline a)dm_x(\underline a)d\nu_{\beta}(x).
\end{equation}
The study of the measures $m_x$ is the principle focus of this article.

Expansions of numbers in non-integer bases have been studied since the 1950s with the work of Renyi \cite{Renyi} and Parry \cite{Parry} who were interested in the properties of the largest $\beta$-expansions of $x$ with respect to the lexicographical ordering, known as the greedy $\beta$-expansion. The dynamics of the associated greedy $\beta$-transformation $x\to\beta x$ (mod 1) have been extensively studied over the last sixty years and are well understood. 

Given $\beta\in(1,2)$, the $\beta$-expansion of $x\in I_{\beta}$ is typically not unique, indeed Lebesgue almost every $x\in I_{\beta}$ has uncountably many $\beta$-expansions \cite{Sidorov1}. There is a substantial amount of recent research trying to understand the properties of the sets $\mathcal E_{\beta}(x)$ for typical $x\in I_{\beta}$, see for example \cite{BakerGrowth, BakerSmall, FengSidorov, CountingBeta} and the references therein. Sets of $\beta$-expansions can be generated dynamically using the Random $\beta$-transformation $K_{\beta}$ of Dajani and Kraaikamp \cite{DKRandom}. We define the random $\beta$-transformation $K_{\beta}:\Sigma\times I_{\beta}\to\Sigma\times I_{\beta}$ by 
\[
K_{\beta}(\omega,x)=\left\lbrace\begin{array}{c c}(\omega,T_0(x))& x \in [0,\frac{1}{\beta})\\ (\sigma(\omega),T_{\omega_1}(x))& x \in [\frac{1}{\beta},\frac{1}{\beta(\beta-1)}]\\ (\omega,T_1(x)) & x \in (\frac{1}{\beta(\beta-1)},\frac{1}{\beta-1}]\end{array}\right. .
\]

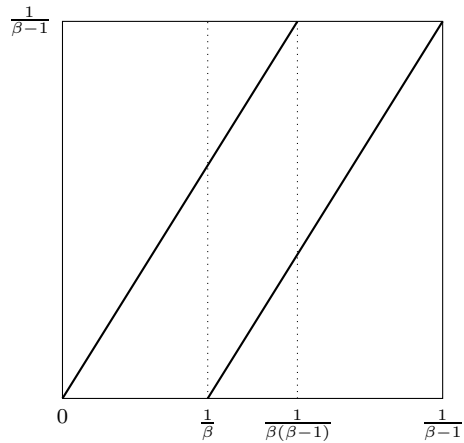
\begin{figure}[ht]
\centering
\begin{tikzpicture}[scale=5]
\draw(0,0)node[below]{\scriptsize 0}--(.382,0)node[below]{\scriptsize$\frac{1}{\beta}$}--(.618,0)node[below]{\scriptsize$\frac{1}{\beta(\beta-1)}$}--(1,0)node[below]{\scriptsize$\frac{1}{\beta-1}$}--(1,1)--(0,1)node[left]{\scriptsize$\frac{1}{\beta-1}$}--(0,.5)--(0,0);
\draw[dotted](.382,0)--(.382,1)(0.618,0)--(0.618,1);
\draw[thick](0,0)--(0.618,1)(.382,0)--(1,1);
\end{tikzpicture}\caption{The projection onto the second coordinate of $K_{\beta}$ for $\beta=\dfrac{1+\sqrt{5}}{2}$}
\end{figure}

Given $x\in I_{\beta}$, $\beta$-expansions of $x$ are generated by choosing some $\omega\in\{0,1\}^{\mathbb N}$ and iterating $K_{\beta}(\omega,x)$. If the $i$th iteration of $K_{\beta}(\omega,x)$ applies $T_0$ to the second coordinate we put $a_i=0$, if it applies $T_1$ to the second coordinate we put $a_i=1$. This generates a sequence $(a_i)$ which is a $\beta$-expansion of $x$, and all $\beta$-expansions of $x$ can be generated this way, see \cite{DKRandom}. 

The measure of maximal entropy of $K_{\beta}$ was studied in \cite{DdV1} and was shown to project to the Bernoulli convolution on its second coordinate. The mapping which takes a pair $(\omega,x)$ to the $\beta$-expansion generated by $(\omega,x)$ is a bijection up to sets of measure zero with respect to the measure of maximal entropy, and thus the system $K_{\beta}$ is a suitable dynamical system for studying both Bernoulli convolutions and sets of $\beta$-expansions.

A full description of the measure of maximal entropy for $K_{\beta}$ was not given in \cite{DdV1}. The authors were able to show that it is not a product measure in general, but the behaviour of this measure on the first coordinate remains unknown in the general case. The measures on $\mathcal E_{\beta}(x)$ introduced in this article allow one to give a full description of the measure of maximal entropy for $K_{\beta}$ in terms of the density of $\nu_{\beta}$ in the case that $\nu_{\beta}$ is absolutely continuous.

The method of coding $\beta$-expansions above gives a bijection (up to sets of measure zero) between $\Sigma$ and $\Sigma\times I_{\beta}$ by associating to a code $(a_i)\in\Sigma$ the corresponding pair $(\omega,x)$. Then the space $\Sigma\times I_{\beta}$ can be seen as a representation of $\Sigma$ for which the complicated projection $\pi_{\beta}$ becomes a simple projection onto the second coordinate, and horizontal fibres can be mapped onto the sets $\mathcal E_{\beta}(x)$. The dynamical system that we build in the next section uses effectively the same idea, except that the sets $\mathcal E_{\beta}(x)$ are represented in a different way which makes invariant measures much easier to study.

\section{A Dynamical System}\label{ExtSec}
We begin by building a dynamical system $(X,\phi,\mu)$ which is measurably isomorphic to the full shift on two symbols (and hence also to the Random $\beta$-transformation), but for which the invariant measure $\mu$ is Lebesgue measure. The sets $\mathcal E_{\beta}(x)$ correspond to vertical slices through the space $X$. 

We assume that $\nu_{\beta}$ is absolutely continuous, and has $\mathcal L^1$ density function $h_{\beta}$. We define the space
\[
X=\{(x,y):x\in I_{\beta}, 0\leq y\leq h_{\beta}(x)\}
\]
and let $\lambda^2$ denote two dimensional Lebesgue measure restricted to $X$.

Now since $\nu_{\beta}$ satisfies the self similarity relation
\[
\nu_{\beta}=\frac{1}{2}\left(\nu_{\beta}\circ T_0 + \nu_{\beta}\circ T_1\right)
\]
we have that $h_{\beta}$ satisfies the relation
\begin{equation}\label{heq}
h_{\beta}(x)=\frac{\beta}{2}\left(h_{\beta}(T_0(x))+h_{\beta}(T_1(x))\right).
\end{equation}
Here we are considering $h_{\beta}$ to be defined on the whole real line, although it takes value $0$ outside of $I_{\beta}$. We partition $X$ into two pieces with non-overlapping interior,
\[
X_0=\{(x,y)\in X: 0\leq y\leq \frac{\beta}{2}h_{\beta}(\beta x)\}
\]
and $X_1=\overline{X\setminus X_0}$. $X_1$ and $X_0$ intersect on a set of Lebesgue measure zero. We define a map $\phi:X\to X$ by
\[
\phi(x,y)=\left\lbrace\begin{array}{cc}\left(\beta x, \frac{2 y}{\beta}\right)&(x,y)\in X_0\\ \left(\beta x-1, \frac{2 y}{\beta}-h_{\beta}(\beta x)\right)&(x,y)\in X_1\end{array}\right. .
\]
The map $\phi$ is well defined except on the intersection of $X_0$ and $X_1$. Because of equation \ref{heq}, we see that $\phi$ maps each of $X_0$ and $X_1$ bijectively onto the whole space $X$ and thus $\phi$ is conjugate to the full shift on two symbols. Furthermore, since $\phi$ stretches the first coordinate by a factor of $\beta$ and stretches the second coordinate by a factor of $\frac{2}{\beta}$, and since each point has exactly two preimages under $\phi$, we see that $\phi$ preserves Lebesgue measure $\lambda^2$. 

The map $\phi$ allows us to assign a unique code $\underline a(x,y)$ to almost every point $(x,y)$ in $X$ by writing 
\[
a_n(x,y)=\left\lbrace\begin{array}{cc}0&\phi^{n-1}(x,y)\in X_0\\1&\phi^{n-1}(x,y)\in X_1\end{array}\right. .
\] There are problems only with boundaries of the partition $X_0,X_1$, as is typical for Markov partition constructions. 

We can describe this coding by a map $P\{0,1\}^{\mathbb N}\to X$. Given a word $a_1\cdots a_n\in\{0,1\}^n$ we let $[a_1\cdots a_n]$ denote the set of sequences $\{\underline x \in\{0,1\}^{\mathbb N}:x_1\cdots x_n=a_1\cdots a_n\}$. We define the set 
\[
[a_1\cdots a_n]_X:=X_{a_1}\cap \phi^{-1}(X_{a_2})\cap\cdots\cap \phi^{-(n-1)}(X_{a_n}).
\]
For each $a_1\cdots a_n\in\{0,1\}^n$ we have $\lambda^2([a_1\cdots a_n]|_X)=2^{-n}$. Then we define $P:\{0,1\}^{\mathbb N}\to X$ by
\[
P(\underline a):=\bigcap_{n=1}^{\infty}[a_1\cdots a_n]_X.
\]
By construction, the coding map $P$ is a measure isomorphism from $(\Sigma,\sigma,m)$ to $(X,\phi,\lambda^2)$. 

\subsection{Pulling Back Lebesgue Measure}
This dynamical system gives rise to a natural measure on the sets $\mathcal E_{\beta}(x)$. Given a code $a_1\cdots a_n\in\{0,1\}^n$ we define
\[
 T_{a_1\cdots a_n}:=T_{a_n}\circ T_{a_{n-1}}\circ\cdots \circ T_{a_1}.
\]
We have that $T_{a_1\cdots a_n}(x)\in I_{\beta}$ if and only if $[a_1\cdots a_n]\cap \mathcal E_{\beta}(x)\neq \phi$, see \cite{DKRandom} for a more detailed description of how to construct $\beta$-expansions. 

Then for $x_0\in I_{\beta}$ we define the fibre
\[
X_{x_0}:=\{(x,y)\in X:x=x_0\}
\]
and see that $P^{-1}(X_x)=\mathcal E_{\beta}(x).$ So we can get a measure on the set $\mathcal E_{\beta}(x)$ by pulling back normalised one dimensional Lebesgue measure on $X_x$. 

This measure can easily be described using $h_{\beta}$. We have that
\[
\phi^n(X_x\cap [a_1\cdots a_n]_X)=X_{T_{a_1\cdots a_n}(x)}.
\]
Then since map $\phi$ expands vertical distances by $\frac{2}{\beta}$, we see that
\begin{eqnarray*}
\lambda(X_x\cap [a_1\cdots a_n]_X)&=&\left(\frac{\beta}{2}\right)^n \lambda(X_{T_{a_1\cdots a_n}(x)})\\
&=&\left(\frac{\beta}{2}\right)^n h_{\beta}(T_{a_1\cdots a_n}(x)),
\end{eqnarray*}
where $\lambda$ denotes one dimensional Lebesgue measure. Summing over all words $a_1\cdots a_n\in\{0,1\}^n$ one recovers equation \ref{heq}. Normalising $\lambda$ to give the fibre total mass $1$, and pulling back to the set $\mathcal E_{\beta}(x)$, we define the measure
\begin{eqnarray*}
m^1_x[a_1\cdots a_n]&:=&\frac{1}{h_{\beta}(x)}\lambda(X_x\cap [a_1\cdots a_n]_X)\\&=&\left(\frac{\beta}{2}\right)^n \frac{h_{\beta}(T_{a_1\cdots a_n}(x))}{h_{\beta}(x)}.
\end{eqnarray*}
The measure $m^1_x$ is a probability measure on $\mathcal E_{\beta}(x)$ defined whenever $\nu_{\beta}$ is absolutely continuous. We prove that it coincides with the measures $m_x$ defined earlier.
\begin{prop}\label{m1Prop}
The measure $m^1_x$ is equal to the measure $m_x$ whenever $m^1_x$ is defined.
\end{prop}
\begin{proof}
We recall the measures $(m_x)_{x\in I_{\beta}}$ were defined as the $\nu_{\beta}$-almost everywhere unique collection of probability measures supported on the sets $\mathcal E_{\beta}(x)$ satisfying equation \ref{disintegration}. The measures $m^{1}_x$ are also probability measures supported on $\mathcal E_{\beta}(x)$, and so we need only to show that they satisfy equation \ref{disintegration} in order to verify that $m_x=m^1_x$. But then, since the map $P$ taking $\Sigma$ to $X$ is a bijection which maps $m$ to two dimensional Lebesgue measure on $X$ and $m^{1}_x$ to one dimensional Lebesgue measure on $X_x$, it is enough to show that
\[
\int_X f(x,y) d\lambda^2=\int_{I_{\beta}}\int_{X_x} f(x,y)d\lambda(y)d\lambda(x)
\]
for each integrable $f$. But this is just the classical Fubini theorem, and so we are done.
\end{proof}

\subsection{Comments on the map $\phi$}
We briefly comment on the relationship between our map $\phi$, the random $\beta$-transformation and the fat baker's transformation of \cite{AlexanderYorke}, since these statements are rather outside of the main thrust of our arguments we make them without proof, but they can easily be deduced by looking at our construction.

Firstly we remark that the system $(X,\phi)$ is in fact rather similar to the random $\beta$-transformation $K_{\beta}$. In fact, if one studies the system $(\{0,1\}^{\mathbb N}\times I_{\beta},K_{\beta},\hat{\nu}_{\beta})$ where $\hat{\nu}_{\beta}$ is the measure of maximal entropy for $K_{\beta}$, then one sees that $(X,\phi,\lambda^2)$ and $(\Omega\times I_{\beta},K_{\beta},\hat{\nu}_{\beta})$ are measurably isomorphic. One can prove this rather cheaply by observing that both systems are measurably isomorphic to the full shift on two symbols coupled with the $(\frac{1}{2},\frac{1}{2})$ Bernoulli measure, but it is quite instructive to build the isomorphism directly. It was an open question stated in \cite{DdV} to determine the behaviour of $\hat{\nu}_{\beta}$ on fibres, the above formula for the measure $m^{1}_x$ answers this question in the case that $\nu_{\beta}$ is absolutely continuous.

There is a simple invertible extension of $(X,\phi,\mu)$ given by defining, $\hat X=X\times[0,1]$, $\hat{\mu}=\lambda^3|_{\hat X}$ where $\lambda^3$ denotes three dimensional Lebesgue measure, and $\hat{\phi}((x,y),z)=(\phi(x,y),\frac{z}{2}+i)$ whenever $(x,y)\in X_i$. The system $(\hat X,\hat{\phi},\hat{\mu})$ is measurably isomorphic to $(\hat{\Sigma},\hat{\sigma},m)$ where $\hat{\Sigma}$ denotes the two sided full shift on $2$-symbols. $\hat{\phi}$ is invertible, and if one projects $\hat{\phi}^{-1}$ onto the first and third coordinates one recovers the fat baker's transformation. It was already known that the fat baker's transformation has the two sided shift on two symbols as an invertible extension, but our map $\hat{\phi}^{-1}$ is perhaps a more interesting natural extension, since it preserves Lebesgue measure and maps down onto the factor system by orthogonal projection.

\section{Hausdorff measure for sets of $\beta$-expansions}\label{HDSec}
In this section we prove results about the Hausdorff measure of sets of $\beta$-expansions. For definitions of Hausdorff measure and Hausdorff dimension see \cite{Falconer}. We endow the space $\{0,1\}^{\mathbb N}$ with metric $d$ defined by
\[
d(\underline a,\underline b)=2^{-\sup\{n:a_1\cdots a_{n}=b_1\cdots b_{n}\}}
\]
if $a_1=b_1$ and $d(\underline a,\underline b)=1$ otherwise. We denote by $|A|$ the diameter of the set $A$, i.e. the supremum of the set of distances between pairs of points in $A$. The diameter of a cylinder set $[a_1\cdots a_n]$ is $2^{-n}$. 

We recall that the density $h_{\beta}$ of $\nu_{\beta}$ is an $\mathcal L^1$ function defined almost everywhere which satisfyies equation \ref{heq}. Since $h_{\beta}$ is defined only almost everywhere, many of our statements about $h_{\beta}$ will hold almost everywhere. In particular, we say that $h_{\beta}$ is bounded if it is essentially bounded, i.e. if there exists a constant $c$ such that $\lambda\{x\in I_{\beta}:h_{\beta}(x)>c\}=0$. We have the following theorem.

\begin{theorem}\label{BetaThm}
The set $\mathcal E_{\beta}(x)$ of $\beta$-expansions of $x$ has positive $\left(\frac{\log(\frac{2}{\beta})}{\log(2)}\right)$-dimensional Hausdorff measure for Lebesgue almost every $x\in I_{\beta}$ if and only if the corresponding Bernoulli convolution $\nu_{\beta}$ is absolutely continuous with bounded density. In this case, normalised Hausdorff measure on the sets $\mathcal E_{\beta}(x)$ coincides with the measures $m_x$. 
\end{theorem}

This theorem is proved using equation \ref{heq}, which allows one a rather simple method of studing the sets $\mathcal E_{\beta}(x)$. We split the theorem into three lemmas. 

\begin{lemma}\label{Lem1}
If the Bernoulli convolution $\nu_{\beta}$ is absolutely continuous with bounded density then the set $\mathcal E_{\beta}(x)$ of $\beta$-expansions of $x$ has positive $\left(\frac{\log(\frac{2}{\beta})}{\log(2)}\right)$-dimensional Hausdorff measure for Lebesgue almost every $x\in I_{\beta}$.
\end{lemma}
\begin{proof}
Let $\tilde{\mathcal U}$ be a countable partition of $\{0,1\}^{\mathbb N}$ by cylinder sets $[a_{1}^i\cdots a_{n_i}^i]$ for $i\in\mathbb N$. We can iterate equation \ref{heq} to write
\[
h_{\beta}(x)=\sum_{a_{1}^i\cdots a_{n_i}^i\in\tilde{\mathcal U}} \left(\frac{\beta}{2}\right)^{n_i} h_{\beta}(T_{a_{1}^i\cdots a_{n_i}^i}(x)).
\]
Since $h_{\beta}(T_{a_{1}^i\cdots a_{n_i}^i}(x))=0$ whenever $T_{a_{1}^i\cdots a_{n_i}^i}\notin I_{\beta}$, we can remove those terms for which $T_{a_{1}^i\cdots a_{n_i}^i}\notin I_{\beta}$, or equivalently $[a_{1}^i\cdots a_{n_i}^i]\cap \mathcal E_{\beta}(x)=\phi$. Then letting
\[
\mathcal U=\{[a_{1}^i\cdots a_{n_i}^i]\in\tilde{\mathcal U}:[a_{1}^i\cdots a_{n_i}^i]\cap \mathcal E_{\beta}(x)\not=\phi\},
\]
the previous equation becomes 
\begin{equation}\label{heq2}
h_{\beta}(x)=\sum_{a_{1}^i\cdots a_{n_i}^i\in\mathcal U} \left(\frac{\beta}{2}\right)^{n_i} h_{\beta}(T_{a_{1}^i\cdots a_{n_i}^i}(x)),
\end{equation}
We stress that, since any open cover of $\mathcal E_{\beta}(x)$ can be obtained by taking a cover of $\{0,1\}^{\mathbb N}$ by cylinder sets and discarding those sets which don't intersect $\mathcal E_{\beta}(x)$, the above equation holds for all covers $\mathcal U$ of $\mathcal E_{\beta}(x)$ by cylinder sets.

Then for any disjoint cover $\mathcal U$ of $\mathcal E_{\beta}(x)$ by cylinder sets we have that
\begin{eqnarray*}
\sum_{a_{1}^i\cdots a_{n_i}^i\in\mathcal U} |[a_{1}^i\cdots a_{n_i}^i]|^{\left(\frac{\log(\frac{2}{\beta})}{\log(2)}\right)}&=&\sum_{a_{1}^i\cdots a_{n_i}^i\in\mathcal U} 2^{-\left(\frac{\log(\frac{2}{\beta})}{\log(2)}\right)n_i}\\
&=& \sum_{a_{1}^i\cdots a_{n_i}^i\in\mathcal U} \left(\frac{\beta}{2}\right)^{n_i}\\
&=& C(\mathcal U)\sum_{a_{1}^i\cdots a_{n_i}^i\in\mathcal U} \left(\frac{\beta}{2}\right)^{n_i}h_{\beta}(T_{a_{1}^i\cdots a_{n_i}^i}(x))\\
&=&C(\mathcal U)h_{\beta}(x),
\end{eqnarray*}
where 

\[C(\mathcal U):=\dfrac{\sum_{a_{1}^i\cdots a_{n_i}^i\in\mathcal U} \left(\frac{\beta}{2}\right)^{n_i}}{\sum_{a_{1}^i\cdots a_{n_i}^i\in\mathcal U} \left(\frac{\beta}{2}\right)^{n_i}h_{\beta}(T_{a_{1}^i\cdots a_{n_i}^i}(x))}.\] The final line here followed from equation \ref{heq2}. 

If $h_{\beta}$ is bounded then $\frac{1}{h_{\beta}(T_{a_{1}^i\cdots a_{n_i}^i}(x))}\geq C>0$ where $C:=\frac{1}{\text{ess}-\sup\{h(x):x\in I_{\beta}\}}$ is independent of $a_{1}^i\cdots a_{n_i}^i$ and $x$. Then $C(\mathcal U)\geq C$ and thus
\[
\sum_{a_{1}^i\cdots a_{n_i}^i\in\mathcal U} |[a_{1}^i\cdots a_{n_i}^i]|^{\left(\frac{\log(\frac{2}{\beta})}{\log(2)}\right)}\geq Ch(x)
\]
for any cover $\mathcal U$ of $\mathcal E_{\beta}(x)$. We conclude that 
\[
\mathcal H^{\left(\frac{\log(\frac{2}{\beta})}{\log(2)}\right)}(\mathcal E_{\beta}(x))\geq Ch_{\beta}(x)>0
\]
for all $x\in I_{\beta}$ such that $h(x)>0$, and in particular for almost all $x\in I_{\beta}$.\end{proof}

We define the measure $m^{2}_x$ on $\mathcal E_{\beta}(x)$ by
\[
m^{2}_x(A)=\frac{\mathcal H^{\left(\frac{\log(\frac{2}{\beta})}{\log(2)}\right)}(A)}{\mathcal H^{\left(\frac{\log(\frac{2}{\beta})}{\log(2)}\right)}(\mathcal E_{\beta}(x))}.
\]
This is well defined for almost every $x\in I_{\beta}$ whenever $h_{\beta}(x)$ is bounded. The second step of the proof of Theorem \ref{BetaThm} is the following.
\begin{lemma}\label{gCor}
The measures $m^{2}_x$ and $m^{1}_x$ are equal whenever they are both defined, i.e. whenever the Bernoulli convolution $\nu_{\beta}$ is absolutely continuous with bounded density.
\end{lemma}
\begin{proof}
We first observe that one has the bound 
\[
\mathcal H^{\left(\frac{\log(\frac{2}{\beta})}{\log(2)}\right)}(\mathcal E_{\beta}(x))\leq 2h_{\beta}(x)
\]
for $x\in I_{\beta}$. To prove this one takes the cover of $\mathcal E_{\beta}(x)$ by all cylinders of depth $n$ which intersect $\mathcal E_{\beta}(x)$. It was proved in \cite{CountingBeta} Lemma 3.4, following a similar argument in Appendix C of \cite{PollicottWeissDimensions}, that
\[
\limsup_{n\to\infty} \left(\frac{\beta}{2}\right)^n|\{a_1\cdots a_n\in\{0,1\}^n:[a_1\cdots a_n]\cap \mathcal E_{\beta}(x)\neq \phi\}|\leq 2h_{\beta}(x).
\]
Then for all $\epsilon>0$ we can by taking $n$ large enough find a cover $\mathcal U$ of $\mathcal E_{\beta}(x)$ by cylinder sets of depth $n$ for which
\[
\sum_{a_{1}\cdots a_{n}\in\mathcal U} |[a_{1}^i\cdots a_{n_i}^i]|^{\left(\frac{\log(\frac{2}{\beta})}{\log(2)}\right)}=|\mathcal U|\left(\frac{\beta}{2}\right)^n\leq 2h_{\beta}(x)+\epsilon.
\]
In particular, we see that 
\[
0<\int_{I_{\beta}}\mathcal H^{\left(\frac{\log(\frac{2}{\beta})}{\log(2)}\right)}(\mathcal E_{\beta}(x)) dx\leq 2.
\]
We define
\[
g(x):=\mathcal H^{\left(\frac{\log(\frac{2}{\beta})}{\log(2)}\right)}(\mathcal E_{\beta}(x)).
\]
Now given a cylinder $[a_1\cdots a_n]$ we have that \[|[a_1\cdots a_n]|=2^{-1}|[a_2\cdots a_n]|=2^{-1}|\sigma[a_1\cdots a_n]|.\] Then given any set $A$ which is contained in either $[0]$ or $[1]$ we have that
\begin{eqnarray*}
\mathcal H^{\left(\frac{\log(\frac{2}{\beta})}{\log(2)}\right)}(A)&=&2^{-\left(\frac{\log(\frac{2}{\beta})}{\log(2)}\right)}\mathcal H^{\left(\frac{\log(\frac{2}{\beta})}{\log(2)}\right)}(\sigma(A))\\&=&\frac{\beta}{2}\mathcal H^{\left(\frac{\log(\frac{2}{\beta})}{\log(2)}\right)}(\sigma(A)).
\end{eqnarray*}

The tree structure of the set of $\beta$-expansions means that
\begin{eqnarray*}
g(x)=\mathcal H^{\left(\frac{\log(\frac{2}{\beta})}{\log(2)}\right)}(\mathcal E_{\beta}(x))&=&\mathcal H^{\left(\frac{\log(\frac{2}{\beta})}{\log(2)}\right)}(\mathcal E_{\beta}(x)\cap[0])+\mathcal H^{\left(\frac{\log(\frac{2}{\beta})}{\log(2)}\right)}(\mathcal E_{\beta}(x)\cap[1])\\
&=& \frac{\beta}{2}\mathcal H^{\left(\frac{\log(\frac{2}{\beta})}{\log(2)}\right)}(\sigma(\mathcal E_{\beta}(x)\cap[0]))+\frac{\beta}{2}\mathcal H^{\left(\frac{\log(\frac{2}{\beta})}{\log(2)}\right)}(\sigma(\mathcal E_{\beta}(x)\cap[1]))\\
&=&\frac{\beta}{2}(g(T_0(x))+g(T_1(x))).
\end{eqnarray*}

Then $g(x)$ is an $\mathcal L^1$ function with positive integral satisfying equation \ref{heq}, and since $\mathcal L^1$ solutions to equation \ref{heq} are unique up to multiplication by constants we see that $g(x)=Kh_{\beta}(x)$ for some constant $K$. In particular, $m^{2}_x$ on $\mathcal E_{\beta}(x)$ assigns mass 
\[
\frac{\left(\frac{\beta}{2}\right)^n g(T_{a_1\cdots a_n}(x))}{g(x)}=\left(\frac{\beta}{2}\right)^n\frac{h_{\beta}(T_{a_1\cdots a_n}(x))}{h_{\beta}(x)}
\]
to cylinder $[a_1\cdots a_n]$ for any choice of $a_1\cdots a_n$, and thus the measures $m^{2}_x$ and $m^{1}_x$ coincide. By proposition \ref{m1Prop} we now have that all three measures $m_{x}, m^{1}_x$ and $m^{2}_x$ coincide when they are defined.
\end{proof}
The proof of the following lemma completes the proof of Theorem \ref{BetaThm}.

\begin{lemma}
If $h_{\beta}(x)$ is unbounded then $\mathcal H^{\left(\frac{\log(\frac{2}{\beta})}{\log(2)}\right)}(\mathcal E_{\beta}(x))=0$ for almost every $x\in I_{\beta}$.
\end{lemma}
\begin{proof}
We stress that, since $h_{\beta}$ is defined only almost everywhere, we take the statement `$h_{\beta}$ is unbounded' to mean that for each $C\in\mathbb R$ the set $A_C:=\{x\in I_{\beta}:h_{\beta}(x)>C\}$ has positive Lebesgue measure.

We begin by supposing that 
\[g(x):=\mathcal H^{\left(\frac{\log(\frac{2}{\beta})}{\log(2)}\right)}(\mathcal E_{\beta}(x))\]
is positive for a positive Lebesgue measure set of $x\in I_{\beta}$. Then $g(x)$ is an $\mathcal L^1$ function of positive integral and the conclusions of Lemma \ref{gCor} hold. 

Now we define the set $B_C$ by
\[
B_C:=\{\underline a\in\Sigma: \pi_{\beta}(\underline a)\in A_C\}
\]
and see that $m(B_C)>0$. We let $B\subset \Sigma$ be the set of sequences $\underline a\in \Sigma$ such that $\sigma^n(\underline a)\in B_C$ infinitely often. Since the system $(\Sigma,\sigma,m)$ is ergodic we see that $m(B)=1$. In particular, $m_x(B^c)=0$ for almost every $x$, giving that \[\mathcal H^{\left(\frac{\log(\frac{2}{\beta})}{\log(2)}\right)}(\mathcal E_{\beta}(x)\cap B^c)=0\]
for almost every $x$. Here we have used the assumption that the conclusions of Lemma \ref{gCor} hold, allowing us to replace $m_x$ with normalised Hausdorff measure.

Now we let $\delta>0$ and let $N\in\mathbb N$ satisfy $2^{-N}<\delta$. For $n\geq N$ we define
\[
A_{n,x}:=\{\underline a\in\mathcal E_{\beta}(x):\sigma^n(\underline a)\in B_C, \underline a\not\in A_{N,x},\cdots A_{n-1,x}\}
\]
Each $A_{n,x}$ consists of a finite number of cylinder sets, and the union of these collections of cylinder sets over $n\geq N$ forms a $\delta$-cover of $\mathcal E_{\beta}(x)\cap B$. Furthermore, on each of these cylinder sets forming $A_{n,x}$ one has $h_{\beta}(\pi_{\beta}(\sigma^n(\underline a))>C$. Then letting $\mathcal U$ be the $\delta$-cover of $\mathcal E_{\beta}(x)\cap B$ using the cylinder sets in $A_{n,x}$ for $n\geq N$, we have
\[
\mathcal C(\mathcal U)<\frac{1}{C}.
\]
Using the final lines of the proof of Lemma \ref{Lem1}, we see that this gives
\begin{eqnarray*}
\mathcal H^{\left(\frac{\log(\frac{2}{\beta})}{\log(2)}\right)}(\mathcal E_{\beta}(x))&=&\mathcal H^{\left(\frac{\log(\frac{2}{\beta})}{\log(2)}\right)}(\mathcal E_{\beta}(x)\cap B^c)+\mathcal H^{\left(\frac{\log(\frac{2}{\beta})}{\log(2)}\right)}(\mathcal E_{\beta}(x)\cap B)\\
&\leq& 0+\frac{h_{\beta}(x)}{C}
\end{eqnarray*}
and since $C$ was arbitrary we are done.
\end{proof}

\section{Equidistribution Results}\label{EqSec}
In this section we use our understanding gained in the last section of the disintegration of $m$ on $\Sigma$ by the sets $\mathcal E_{\beta}(x)$ to turn some results of \cite{LPS} into equidistribution results for sets of $\beta$-expansions. It is likely that, by suitably adapting the results of \cite{LPS} to the case of projecting and slicing self similar sets, one could prove similar results for equidistribution of slices of fractals. Our main question is the following.

{\bf Question:} What can one say about the distribution of the multisets
\[
\mathcal O^n(x):=\{T_{a_1\cdots a_n}(x):[a_1\cdots a_n]\cap \mathcal E_{\beta}(x)\not=\phi\},
\]
where the multiplicity of $y\in\mathcal O^n(x)$ is defined as being equal to the number of words $a_1\cdots a_n$ for which $T_{a_1\cdots a_n}(x)=y$. In particular, what is the relationship between the limiting distribution of $\mathcal O^n(x)$ for typical $x$ and the question of the absolute continuity of $\nu_{\beta}$?

If $\beta$ is non algebraic then there do not exist words $a_1\cdots a_n\neq b_1\cdots b_n\in\{0,1\}^n$ such that $T_{a_1\cdots a_n}(x)=T_{b_1\cdots b_n}(x)$, and thus the multiplicity of elements of $\mathcal O^n(x)$ is always equal to $1$.

We define
\[
\mathcal N_n(x;\beta):=|\mathcal O^n(x)|=\left|\{a_1\cdots a_n\in\{0,1\}^n:T_{a_1\cdots a_n}(x)\in I_{\beta}\}\right|
\]
In \cite{CountingBeta} we were able to link the growth rate of $\mathcal N_n(x;\beta)$ for typical $x\in I_{\beta}$ with the question of the absolute continuity of $\nu_{\beta}$. In particular we defined
\[
\overline f(x):=\limsup_{n\to\infty}\left(\frac{\beta}{2}\right)^n\mathcal N_n(x;\beta)
\]
and $\underline f(x)$ as above but with the $\limsup$ replaced by a $\liminf$. We proved that if either $\overline f$ or $\underline f$ were $\mathcal L^1$ functions with positive integral then $\nu_{\beta}$ is absolutely continuous. We conjectured that for absolutely continuous $\nu_{\beta}$ one has $\overline f =\underline f$.

We are interested in the extent to which equidistribution of $\mathcal O^n(x)$ is implied by the absolute continuity of $\nu_{\beta}$. We are not able to answer this question, but we can at least show that equidistribution is typical for $\beta\in(1,\sqrt{2})$. The following theorem is a restatement in our language of Theorem 1.2 of \cite{LPS}.

\begin{theorem}\label{LPSThm}[Lindenstrauss, Peres and Schlag]
For almost every $\beta\in(1,2)$, for each $a_1\cdots a_m\in\{0,1\}^m$ and for almost every $x\in I_{\beta}$ we have that \[m_x\{\underline w \in\mathcal E_{\beta}(x):\sigma^n(\underline w)\in[a_1\cdots a_m]\}\to_{n\to\infty} 2^{-m}.\]
\end{theorem}
Given an interval $A\subset I_{\beta}$ and $\epsilon>0$, we can approximate $A$ below by a finite collection $\mathcal U_1$ of disjoint cylinder sets such that \[\sum_{[a_1\cdots a_m]\in\mathcal U_1}m[a_1\cdots a_m]>\nu_{\beta}(A)-\epsilon\] and $\pi_{\beta}[a_1\cdots a_m]\subset A$ for each $[a_1\cdots a_m]\in\mathcal U_1$. Similarly we can approximate $A$ from above with a collection $\mathcal U_2$ of cylinder sets such that \[\sum_{[a_1\cdots a_m]\in\mathcal U_2}m[a_1\cdots a_n]<\nu_{\beta}(A)+\epsilon\] and \[\pi_{\beta}(\underline a)\in A\implies \underline a\in \bigcup_{[a_1\cdots a_m]\in\mathcal U_2}[a_1\cdots a_m].\]
Then an immediate corollary to Theorem \ref{LPSThm} is that for almost every $\beta\in(1,2),x\in I_{\beta}$ and for each interval $A\subset I_{\beta}$ we have that
\[
m_x\{\underline w \in\{0,1\}^{\mathbb N}:\pi_{\beta}(\sigma^n(\underline w))\in A\}\to_{n\to\infty} \nu_{\beta}(A).
\]
Equivalently,
\begin{cor}\label{LPSCor}
For almost every $\beta\in(1,2)$ and for almost every $x\in I_{\beta}$ the probability measures
\[
\nu_{n,x}:=\sum_{a_1\cdots a_n\in\{0,1\}^n}\delta_{T_{a_1\cdots a_n}(x)}m_x[a_1\cdots a_n]
\]
converge weak$^*$ to $\nu_{\beta}$ as $n\to\infty$. 
\end{cor}

This is an equidistribution result stated in terms of conditional measures, and was well suited to the purposes of \cite{LPS} as it allowed them to answer an old question of Sinai and Rokhlin about conditional entropy. However if one is interested in the distribution of the sets $\mathcal O^n(x)$ it would be more natural to seek equidistribution results that did not depend on the conditional measures $m_x$. We define probability measures
\[
\mu_{n,x}:=\frac{1}{\mathcal N_n(x;\beta)}\sum_{y\in\mathcal O^n(x)}\delta_y
\]
and have the following theorem.
\begin{theorem}\label{EquiThm}
For almost every $\beta\in(1,\sqrt{2})$, we have that
\[
\mu_{n,x}\to\lambda|_{I_{\beta}}
\]
weakly as $n\to\infty$ for almost every $x\in I_{\beta}$.
\end{theorem}
We conjecture that the conclusions of this theorem hold whenever $\nu_{\beta}$ is absolutely continuous.\footnote{Some progress in this direction was announced by C. Bandt at a recent conference in Hong Kong, at time of writing no preprint is availible.}

Given the description of the measures $m_x$ in the earlier sections, it seems natural that Theorem \ref{EquiThm} follows from Corollary \ref{LPSCor}. In some sense, all we are doing is dividing by the density $h_{\beta}(x)$ to turn $\nu_{\beta}$ into $\lambda|_{I_{\beta}}$ on the right hand side and $\nu_{n,x}$ into $\mu_{n,x}$ on the left. However we have to do this formally, and also to be careful to ensure that too much of $\mu_{n,x}$ isn't concentrated at the edges of $I_{\beta}$. We prove Theorem \ref{EquiThm}.

\begin{proof}
We assume that $\beta$ is such that $\nu_{\beta}$ is absolutely continuous with continuous density $h_{\beta}$ which is strictly positive on $(0,\frac{1}{\beta-1})$ and that $\beta$ satisfies the conclusions of Corollary \ref{LPSCor}. This holds for almost every $\beta\in(1,\sqrt{2})$, the fact that $h_{\beta}$ is strictly positive on $(0,\frac{1}{\beta-1})$ for almost every $\beta\in(1,\sqrt{2})$ was proved in \cite{JordanShmerkinSolomyak}. 

Now let $A\subset I_{\beta}$ be such that there exists a constant $h_{\beta}(A)$ for which 
\begin{equation}\label{hratio}
h_{\beta}(A)(1-\epsilon)<h_{\beta}(x)<h_{\beta}(A)(1+\epsilon)
\end{equation}
for each $x\in A$. Then we have that
\begin{eqnarray*}
\sum_{a_1\cdots a_n\in\{0,1\}^n}m_x[a_1\cdots a_n]\chi_A(T_{a_1\cdots a_n}(x))&=&\sum_{a_1\cdots a_n\in\{0,1\}^n}\left(\frac{\beta}{2}\right)^n\frac{h_{\beta}(T_{a_1\cdots a_n}(x))}{h_{\beta}(x)}\chi_A(T_{a_1\cdots a_n}(x))\\
&\leq& \left(\frac{\beta}{2}\right)^n \frac{h_{\beta}(A)(1+\epsilon)}{h_{\beta}(x)}|\mathcal O^n(x)\cap A|.
\end{eqnarray*}
Now Corollary \ref{LPSCor} says that 
\[
\sum_{a_1\cdots a_n\in\{0,1\}^n}m_x[a_1\cdots a_n]\chi_A(T_{a_1\cdots a_n}(x))\to_{n\to\infty} \nu_{\beta}(A)\geq \lambda(A)h_{\beta}(A)(1-\epsilon)
\]
as $n\to\infty$. Then using the fact that $m_x=m^1_x$ we have for sufficiently large $n$ that

\begin{eqnarray*}
\lambda(A)h_{\beta}(A)(1-\epsilon)^2\leq \left(\frac{\beta}{2}\right)^n \frac{h_{\beta}(A)(1+\epsilon)}{h_{\beta}(x)}|\mathcal O^n(x)\cap A|,
\end{eqnarray*}
giving
\[
|\mathcal O^n(x)\cap A|\geq \lambda(A)h_{\beta}(x)\left(\frac{2}{\beta}\right)^n\frac{(1-\epsilon)^2}{1+\epsilon}.
\]
Similarly,
\begin{equation}\label{OnxUB}
|\mathcal O^n(x)\cap A|\leq \lambda(A)h_{\beta}(x)\left(\frac{2}{\beta}\right)^n\frac{(1+\epsilon)^2}{1-\epsilon}.
\end{equation}

The following proposition will be proved at the end of this theorem.
\begin{prop}\label{GrowthProp}
For almost all $\beta\in(1,\sqrt{2})$ and for almost all $x\in I_{\beta}$ we have that
\[
\lim_{n\to\infty} \left(\frac{\beta}{2}\right)^n |\mathcal O^n(x)|\to h_{\beta}(x)
\]
\end{prop}
Then we see that, since $\epsilon$ was arbitrary in equation \ref{OnxUB},
\begin{eqnarray*}
\mu_n(x)(A)=\frac{\left(\frac{\beta}{2}\right)^n|\mathcal O^n(x)\cap A|}{\left(\frac{\beta}{2}\right)^n|\mathcal O^n(x)|}\to\lambda(A)\frac{h_{\beta}(x)}{h_{\beta}(x)}=\lambda(A)
\end{eqnarray*}
as $n\to\infty$. Now any interval $B\subset (\delta,\frac{1}{\beta-1}-\delta)$ can be written as a union of intervals $A_i$ for which there is a constant $h_{\beta}(A)$ such that equation \ref{hratio} holds, and so the proof of Theorem \ref{EquiThm} will be complete once we have proved that not too much mass is concentrated in sets $[0,\delta)$. This is included in the proof of Proposition \ref{GrowthProp}. \end{proof}
It remains only to prove Proposition \ref{GrowthProp}. This is an interesting proposition in its own right showing that Conjecture 1 of \cite{CountingBeta} holds at least for almost every $\beta\in(1,\sqrt{2})$. It was conjectured in \cite{CountingBeta} that this proposition holds for almost all $x\in I_{\beta}$ for all $\beta$ such that $\nu_{\beta}$ is absolutely continuous, this conjecture remains open. A similar question was asked in \cite{Girgensohn} relating to solutions of the Schilling equation, which share many similarities with Bernoulli convolutions.
\begin{proof}
We assume that $\beta$ is non-algebraic and that the conditions of the previous theorem hold, i.e. that $\nu_{\beta}$ is absolutely continuous with continuous density $h_{\beta}$ which is strictly positive on $(0,\frac{1}{\beta-1})$ and that $\beta$ satisfies the conclusions of Corollary \ref{LPSCor}. This holds for almost every $\beta\in(1,\sqrt{2})$. Since $h_{\beta}(x)>0$ on the interior of $I_{\beta}$ and $h_{\beta}$ is uniformly continuous we have that for any $\delta>0$ we can cover $(\delta, \frac{1}{\beta-1}-\delta)$ with intervals $A_1\cdots A_k$ such that there exists $h_{\beta}(A_i)$ satisfying equation \ref{hratio}. 

We first suppose for some $n_k\to\infty$ we have that too much of $|\mathcal O^{n_k}(x)|$ is concentrated in $[0,\delta)$ for some $\delta>0$. To be concrete, we suppose that there exists some $K>\frac{2}{2-\beta}$ for which
\[
|\mathcal O^{n_k}(x)\cap[0,\delta)|>h_{\beta}(x)\left(\frac{2}{\beta}\right)^{n_k}K\delta.
\]
But
\[
|\mathcal O^{n_k}(x)\cap[0,\delta)|=\left|\mathcal O^{n_k-1}(x)\cap\left[0,\frac{\delta}{\beta}\right)\right|+\left|\mathcal O^{n_k-1}(x)\cap\left[\frac{1}{\beta},\frac{1+\delta}{\beta}\right)\right|,
\]
and for for any $\epsilon>0$ there exists $N\in\mathbb N$ such that
\begin{equation}\label{EQ6}
\left|\mathcal O^{n}(x)\cap\left[\frac{1}{\beta},\frac{1+\delta}{\beta}\right)\right|<(1+\epsilon)h_{\beta}(x)\left(\frac{2}{\beta}\right)^{n}\frac{\delta}{\beta}
\end{equation}
for all $n>N$ by equation \ref{OnxUB}. Then we must have that
\begin{eqnarray*}
\left|\mathcal O^{n_k-1}(x)\cap\left[0,\frac{\delta}{\beta}\right)\right|&>&h(x)\left(\frac{2}{\beta}\right)^{n_k}K\delta-(1+\epsilon)h_{\beta}(x)\left(\frac{2}{\beta}\right)^{n_k-1}\frac{\delta}{\beta}\\
&=& h_{\beta}(x)\left(\frac{2}{\beta}\right)^{n_k-1}K\delta\left(\frac{2}{\beta}-\frac{1+\epsilon}{\beta K}\right)\\
&>&h_{\beta}(x)\left(\frac{2}{\beta}\right)^{n_k-1}K\delta,
\end{eqnarray*}
since $\left(\frac{2}{\beta}-\frac{1+\epsilon}{\beta K}\right)>1$ for sufficiently small $\epsilon$ by our choice of $K$. We iterate this equation to get 
\[
|\mathcal O^{n_k-m}(x)\cap[0,\frac{\delta}{\beta^m})|\leq\left|\mathcal O^{n_k-m-1}(x)\cap\left[0,\frac{\delta}{\beta^{m+1}}\right)\right|+\left|\mathcal O^{n_k-1}(x)\cap\left[\frac{1}{\beta},\frac{1+\delta}{\beta}\right)\right|,
\]
where we are using the interval $\left[\frac{1}{\beta},\frac{1+\delta}{\beta}\right)$ rather than $\left[\frac{1}{\beta},\frac{1+\delta}{\beta^m}\right)$ on the right because it allows us to use equation \ref{EQ6}. Iterating this equation to stage $n_k-N$ gives
\[
|\mathcal O^{N}(x)\cap[0,\frac{\delta}{\beta^{n_k-N}})|>h_{\beta}(x)\left(\frac{2}{\beta}\right)^{N}K\delta.
\]
Taking $n_k\to\infty$ we see that the multiset $\mathcal O^N(x)$ must contain the value $0$ multiple times. Since we have assumed that $\beta$ is non-algebraic, this is a contradiction. By symmetry, the same arguments show that not too much of $\mathcal O^n(x)$ can be concentrated in $[\frac{1}{\beta-1}-\delta, \frac{1}{\beta-1}]$.

Then building on the proof of the previous theorem, we can cover $(\delta, \frac{1}{\beta-1}-\delta)$ with intervals $A_i$ upon which $h_{\beta}(x)$ is constant up to multiplicative error $\epsilon$. Summing over $A_i$ and using the bounds in the proof of the previous theorem gives
\begin{eqnarray*}
|\mathcal O^n(x)\cap(\delta, \frac{1}{\beta-1}-\delta)|&=& \sum_{i=1}^k |\mathcal O^n(x)\cap A_i|\\
&\geq& \left(h_{\beta}(x)\left(\frac{2}{\beta}\right)^n\frac{(1-\epsilon)^2}{1+\epsilon}\right)\left(\sum_{i=1}^k \lambda(A_i)\right)\\
&\geq& h_{\beta}(x)\left(\frac{2}{\beta}\right)^n\left(\frac{(1-\epsilon)^2}{1+\epsilon}(1-2\delta)\right).
\end{eqnarray*}
Then
\begin{eqnarray*}
|\mathcal O^{n}(x)|&=&|\mathcal O^{n}(x)\cap[0,\delta)|+\left|\mathcal O^{n}(x)\cap\left(\frac{1}{\beta-1}-\delta,\frac{1}{\beta-1}\right]\right|+\sum_{i=1}^k|\mathcal O^n(x)\cap A_i|\\
&\leq& h_{\beta}(x)\left(\frac{2}{\beta}\right)^n\left(2\delta\frac{2}{2-\beta}+\frac{(1-\epsilon)^2}{1+\epsilon}(1-2\delta)\right)
\end{eqnarray*}

Then since $\delta$ and $\epsilon$ were arbitrary we see that
\[
\lim_{n\to\infty}\left(\frac{\beta}{2}\right)^n|\mathcal O^n(x)|=h_{\beta}(x)
\]
as required.\end{proof}

\subsection{Absolute Continuity from Strong Equidistribution}
For a partial converse, we show that if the sets $\mathcal O^n(x)$ equidistribute in a strong sense for almost every $x$ then $\nu_{\beta}$ is absolutely continuous. We note that the normalised Lebesgue measure of the switch region $S:=[\frac{1}{\beta},\frac{1}{\beta(\beta-1)}]$ is
\[
\left(\frac{1}{\beta(\beta-1)}-\frac{1}{\beta}\right)(\beta-1)=\frac{2}{\beta}-1
\]
Then if the measures $\mu_{n,x}$ converge weak$^*$ to normalised Lebesgue measure we would expect
\[
k_n(x):=\frac{\beta}{2}\left(\mu_{n,x}(S)+1\right)
\]
to converge to $1$. The following proposition shows that fast equidistribution of $\mathcal O^n(x)$ implies the absolute continuity of $\nu_{\beta}$.
\begin{prop}
Suppose that $\prod_{n=1}^{\infty}(k_n(x))$ is an $\mathcal L^1$ function of $x$. Then $\nu_{\beta}$ is absolutely continuous.
\end{prop}
In particular, if the sets $\mathcal O^n(x)$ equidistribute fast enough and uniformly across $x$ then $k_n(x)$ tend to $1$ quickly and so the conditions of the theorem will be satisfied and $\nu_{\beta}$ will be absolutely continuous.
\begin{proof}
We see that we have a choice of the value of $a_{n+1}$ if and only if $T_{a_1\cdots a_n}(x)\in S$, otherwise there is a unique $a_{n+1}$ such that $T_{a_1\cdots a_{n+1}}(x)\in I_{\beta}$. Then have that 
\begin{eqnarray*}
\left(\frac{\beta}{2}\right)^{n+1}\mathcal N_{n+1}(x)=\left(\frac{\beta}{2}\right)^{n+1}|\mathcal O^{n+1}(x)|&=&\left(\frac{\beta}{2}\right)^{n+1}(|\mathcal O^n(x)|+|\mathcal O^n(x)\cap S|)\\
&=&\left(\frac{\beta}{2}\right)^{n+1}|\mathcal O^n(x)|.(1+\mu_{n,x}(S))\\
&=&\left(\frac{\beta}{2}\right)^{n+1}\prod_{i=1}^n(1+\mu_{i,x}(S))\\
&=&\left(\frac{\beta}{2}\right)^{n+1}\left(\frac{2}{\beta}\right)^n\prod_{i=1}^nk_i(x)\\
&=&\left(\frac{\beta}{2}\right)\prod_{i=1}^nk_i(x),
\end{eqnarray*}
which converges to an $\mathcal L^1$ function by the assumptions of our proposition. But the main theorem of \cite{CountingBeta} states that if $f_n(x):=\left(\frac{\beta}{2}\right)^{n}\mathcal N_{n}(x)$ converges to an $\mathcal L^1$ function then $\nu_{\beta}$ is absolutely continuous. \end{proof}

\section{Slicing Fractal Sets}\label{FractalSec}
We now turn to the question of disintegrating Hausdorff measure for self similar sets. The techniques that we used in the symbolic case can be combined with a few technical lemmas to show that slices through certain fractals have positive Hausdorff measure if and only if the corresponding projected measures are absolutely continuous with bounded density. We begin with some background on fractals.

Let $E\subset \mathbb R^n$ be a self-similar set without rotations, that is, a set satisfying
\[
E=\bigcup_{i=1}^l S_i(E)
\]
where the maps $S_i:\mathbb R^n\to\mathbb R^n$ are of the form $S_i(x)=\lambda_i(x)+d_i$ for some $\lambda_i\in(0,1), d_i\in\mathbb R^n$. We further suppose that our iterated function system satisfies the open set condition, i.e. that there is a non-empty open set $V\subset \mathbb R^n$ such that $V\supset \bigcup_{i=1}^l S_i(V)$ where the union is disjoint. Then $E$ has Hausdorff dimension $s$ satisfying
\[
\sum_{i=1}^l\lambda_i^s=1.
\]
Furthermore, the $s$-dimensional Hausdorff measure $\nu$ on $E$ is positive and finite and satisfies the self similarity relation
\[
\nu(A)=\sum_{i=1}^l \lambda_i^s\nu(\tilde T_i(A)),
\]
where $\tilde T_i(x):=S_i^{-1}(x)$. The open set condition implies that for almost every $x\in E$ there is a unique code 
$\underline a\in\Sigma:=\{0,\cdots,l\}^{\mathbb N}$ such that \[x\in [a_1\cdots a_n]_E:=S_{a_n}\circ S_{a_{n-1}}\circ\cdots\circ S_{a_1}(E)\] for each $n\in\mathbb N$. We call $\underline a$ the address of $x$.

We let $\pi_{\theta}$ denote orthogonal projection of $\mathbb R^n$ down a line $l_{\theta}$ through the origin at angle $\theta=(\theta_1,\cdots,\theta_{n-1})$. We let $\nu_{\theta}=\nu\circ\pi_{\theta}^{-1}$. Then $\nu_{\theta}$ satisfies the relation
\[
\nu_{\theta}(A)=\sum_{i=1}^l \lambda_i^s \nu(T_{i}(A))
\]
where $T_{i}(x)=\lambda_i^{-1}x-\pi_{\theta} (a_i)$ is the projection of the map $\tilde T_i$ under $\pi_{\theta}$. 

Now if $s>1$ then the Marstrand projection theorem says that for almost every value of $\theta$ the projection $\nu_{\theta}$ is absolutely continuous. The Marstrand slicing theorem says that for almost every $\theta$ and for almost every $x\in E_{\theta}$ the slice $E_{\theta,x}$ has Hausdorff dimension $s-1$ and has finite $(s-1)$-dimensional Hausdorff measure. We refer the reader to \cite{Falconer, ClimenhagaPesin} for proofs and discussions of the Marstrand slicing and projection theorems.

We let $h_{\theta}:\mathbb R\to\mathbb R^{+}$ be the density of $\nu_{\theta}$ if it exists, $h_{\theta}$ takes value $0$ outside of $E_{\theta}$. Then differentiating the self similarity equation for $\nu_{\theta}$ we see that
\begin{equation}\label{heqf}
h_{\theta}(x)=\sum_{i=1}^l \lambda_i^{s-1} h_{\theta}(T_{i}(x)),
\end{equation}
where we have used that the derivative of each $T_i$ is equal to $\lambda_i$. For $a_1\cdots a_n\in\{0,\cdots l\}^n$ we define the set
\[
[a_1\cdots a_n]_{E_{\theta,x}}:=E_{\theta,x}\cap (S_{a_n}\circ S_{a_{n-1}}\circ\cdots\circ S_{a_1}(E))
\]
Equation \ref{heqf} is our main tool in the proof of our theorem about the positivity of Hausdorff measure of slices through self similar sets. The proof of Theorem \ref{FractalThm} is similar to that of Theorem \ref{BetaThm}, but we require some extra lemmas to estimate the diameter of sets $[a_1\cdots a_n]_{E_{\theta,x}}$ because, unlike in the symbolic case, this diameter is not purely determined by the length of the word $a_1\cdots a_n$. We let $|A|$ denote the Euclidean diameter of a set $A$. This issue with diameters also means that we need the following condition:

\begin{defn}\label{DeltaDefn}
We say that a self similar set $E$ satisfies the slice coding condition if for all $\theta$ there exists a constant $\delta$ such that for all $x\in \pi_{\theta}(E)$ we have that either $|E_{\theta,x}|>\delta$ or $E_{\theta,x}\subset [a]_E$ for some $a\in\{1,\cdots,l\}$.
\end{defn}
Here ${[}a_1\cdots a_m{]}_E := S_{a_1}\circ\cdots\circ S_{a_m}(E)$. We suspect that all self similar sets where the self similarities do not contain rotations satisfy this condition, but we are unable to prove this. We assume for the rest of the article that the slice coding condition is satisfied.




\begin{lemma}\label{hRatio}
Suppose that $h_{\theta}$ is bounded. Then there exists a constant $C$ such that $\frac{h_{\theta}(x)}{|E_{\theta,x}|^{s-1}}<C$ for all $x\in\pi_{\theta}(E)$.
\end{lemma}

\begin{proof}
First let $C:=\sup \{\frac{h_{\theta}(x)}{|E_{\theta,x}|^{s-1}}:|E_{\theta,x}|\geq\delta\}$ where $\delta$ was defined in the Definition $\ref{DeltaDefn}$. The fact that $h_{\theta}$ is bounded implies that $C$ is finite.

Now suppose that $0<|E_{\theta,x}|<\delta$. Then since $E_{\theta,x}$ satisfies the slice coding condition, there exists a unique $n\in\mathbb N$ and word $a_1\cdots a_n$ such that $E_{\theta,x}\subset [a_1\cdots a_n]_{E_{\theta,x}}$ but $E_{\theta,x}\not\subset [a_1\cdots a_{n+1}]_{E_{\theta,x}}$ for any choice of $a_{n+1}\in \{1,\cdots, n\}$. In particular, we have that $E_{\theta,T_{a_1\cdots a_n}(x)}\not\subset [a_{n+1}]_{E_{\theta,T_{a_1\cdots a_n}(x)}}$ for any choice of $a_{n+1}$, and so $|E_{\theta,T_{a_1\cdots a_n}(x)}|>\delta$.

Then using equation \ref{heqf} we have that
\[
h_{\theta}(x)=(\lambda_{a_n}\lambda_{a_{n-1}}\cdots \lambda_{a_1})^{s-1}h_{\theta}(T_{a_1\cdots a_n}(x)).
\]
By the self similarity of $E$ we have that
\[
|E_{\theta,x}|=\lambda_{a_n}\lambda_{a_{n-1}}\cdots \lambda_{a_1}|E_{\theta, T_{a_1\cdots a_n}(x)}|.
\]
Then
\[
\dfrac{h_{\theta}(x)}{|E_{\theta,x}|^{s-1}}=\dfrac{h_{\theta}(T_{a_1\cdots a_n}(x))}{|E_{\theta, T_{a_1\cdots a_n}(x)}|^{s-1}}\leq C
\]
where the final inequality follows from the definition of $C$ because $|E_{\theta, T_{a_1\cdots a_n}(x)}|\geq\delta$.
\end{proof}

Then following the proof of Theorem \ref{BetaThm}, we have the following theorem.

\begin{theorem}\label{FractalThm}
Suppose that $E$ is the attractor of an IFS without rotations satisfying the open set condition and definition \ref{DeltaDefn}. We further assume that the projection of Hausdorff measure on $E$ onto the line at angle $\theta$ through the origin is absolutely continuous with bounded density. Then $\mathcal H^{s-1}(E_{\theta,x})>0$ for $\nu_{\theta}$-almost every $x\in\pi_{\theta}(E)$.
\end{theorem}
\begin{proof}
We recall that Hausdorff measure is defined as the limit as $\delta\to 0$ of the infimum over all $\delta$-coverings $\mathcal U=\{\mathcal U_i\}$ of the quantity $\sum_{i=1}^{\infty} |\mathcal U_i|^s$. It is enough to consider coverings which are unions of cylinder sets $[a_1\cdots a_n]_{E_{\theta,x}}$. 
Then we have that
\[
|[a_1\cdots a_n]_{E_{\theta,x}}|=(\lambda_{a_1}\lambda_{a_2}\cdots\lambda_{a_n}).|E_{\theta,T_{a_1\cdots a_n}(x)}|
\]
Following our proof of Theorem \ref{BetaThm}, we have
\begin{eqnarray*}
h_{\theta}(x)&=&\sum_{[a_1\cdots a_n]_{E_{\theta,x}}\in\mathcal U} (\lambda_{a_1}\lambda_{a_2}\cdots\lambda_{a_n})^{s-1}h_{\theta}(T_{a_1\cdots a_n}(x))\\
&=& \sum_{[a_1\cdots a_n]_{E_{\theta,x}}\in\mathcal U} |[a_1\cdots a_n]_{E_{\theta,x}}|^{s-1} \left(\dfrac{\lambda_{a_1}\lambda_{a_2}\cdots\lambda_{a_n}}{|[a_1\cdots a_n]_{E_{\theta,x}}|}\right)^{s-1}h_{\theta}(T_{a_1\cdots a_n}(x))\\
&=& \sum_{[a_1\cdots a_n]_{E_{\theta,x}}\in\mathcal U} |[a_1\cdots a_n]_{E_{\theta,x}}|^{s-1} \dfrac{h_{\theta}(T_{a_1\cdots a_n}(x))}{(|E_{\theta,T_{a_1\cdots a_n}(x)}|)^{s-1}}\\
&=&C(\mathcal U)\sum_{[a_1\cdots a_n]_{E_{\theta,x}}\in\mathcal U} |[a_1\cdots a_n]_{E_{\theta,x}}|^{s-1},
\end{eqnarray*}
where $C(\mathcal U)$ is a weighted average of the values of $\dfrac{h_{\theta}(T_{a_1\cdots a_n}(x))}{(|E_{\theta,T_{a_1\cdots a_n}(x)}|)^{s-1}}$ over different $a_1\cdots a_n\in\mathcal U$. In particular, since $C(\mathcal U)<C$ for all covers $\mathcal U$, where $C$ is the constant defined in Lemma \ref{hRatio}, we see that
\[
\sum_{[a_1\cdots a_n]_{E_{\theta,x}}\in\mathcal U} |[a_1\cdots a_n]_{E_{\theta,x}}|^{s-1}>\frac{h_{\theta}(x)}{C}
\]
for each cover $\mathcal U$ of $E_{\theta,x}$, finally yielding that
\[
\mathcal H^{s-1}(E_{\theta,x})>\frac{h_{\theta}(x)}{C}
\]
which is positive on a set of $x$ of positive Lebesgue measure.
\end{proof}
\subsection{Further Fractal Results}
In this section we outline how the remaining results of sections \ref{ExtSec} and \ref{HDSec} transfer over to the fractal case. We have done the difficult part (turning Lemma \ref{Lem1} into Theorem \ref{FractalThm}), the remaining results are extremely straightforward and we do not cover them in detail.

First we remark that one can build a dynamical system analagous to that of section \ref{ExtSec} related to the set $E$. We define the space
\[
X_{\theta}:=\{(x,y)\in\mathbb R^2: x\in E_{\theta}, 0\leq y\leq h_{\theta}(x)\}.
\]

The self-similarity equation \ref{heq} for $h_{\beta}$ is directly analogous to the self similarity equation \ref{heqf} for $h_{\theta}$, and using the transformations $T_1\cdots T_l$ one can partition $X_{\theta}$ into subsets $X_{\theta}^1,\cdots X_{\theta}^l$ in the same way that $X$ was partitioned into $X_1, X_2$. We define a dynamical system on $X_{\theta}$ using the transformations $T_1\cdots T_l$ in the same was as was done in the construction of $\phi$ in section \ref{ExtSec}, and this induces a coding of elements of $X_{\theta}$. By mapping elements of $E$ to the elements of $X_{\theta}$ which have the same code, one has an isomorphism (up to sets of measure zero) between $(E, \mathcal H^s|_{E})$ and $(X_{\theta}, \lambda^2|_{X_{\theta}})$ where $\lambda^2$ is two dimensional Lebesgue measure. 

Now one can define a measure $\mu^{1}_x$ on the slice $E_{\theta,x}$ by pulling back normalised Lebesgue measure from the fibres $\{(x,y):0\leq y\leq h_{\theta}(x)\}$. This gives
\[
\mu^1_x([a_1\cdots a_n]_{E_{\theta,x}}):=\dfrac{(\lambda_{a_1}\cdots \lambda_{a_n})^{s-1} h_{\theta}(T_{a_1\cdots a_n}(x))}{h_{\theta}(x)}
\]
for $a_1\cdots a_n\in\{0,\cdots l\}^n$.

By the same Fubini argument given in the proof of Proposition \ref{m1Prop} we see that the probability measures $\mu^1_x$ disintegrate Hausdorff measure $\mathcal H^s$ on $E$.

We now wish to show that this disintegration coincides with normalised Hausdorff measure on slices. We define $\mu^2_x$ on sets $E_{\theta,x}$ by
\[
\mu_2(A)=\dfrac{\mathcal H^{s-1}(A)}{\mathcal H^{s-1}(E_{\theta,x})} 
\]
for $A\subset E_{\theta,x}$, which is well defined $\nu_{\theta}$ almost everywhere by Theorem \ref{FractalThm}. In \cite{Marstrand1954b}, Marstrand proved that
\[
\mathcal H^s(E)\geq \int_{\pi_{\theta}(E)}\mathcal H^{s-1}(E_{\theta,x})dx.
\]
Combined with our previous theorem it shows that, under the conditions of Theorem \ref{FractalThm},
\[
g(x):=\mathcal H^{s-1}(E_{\theta,x})
\]
is an $\mathcal L^1$ function with positive integral. But then following the proof of Corollary \ref{gCor}, we see that the function $g$ satisfies equation \ref{heqf}, and therefore there is a constant $K(\theta)$ such that
\[
g(x)=K(\theta)h_{\theta}(x).
\]
Finally, we note that $[a_1\cdots a_n]_{E_{\theta,x}}$ is a copy of $E_{\theta, T_{a_1\cdots a_n}(x)}$ scaled down by factor $\lambda_{a_1}\cdots \lambda_{a_n}$, and so we have that
\[
\mathcal H^{s-1}([a_1\cdots a_n]_{E_{\theta,x}})=(\lambda_{a_1}\cdots \lambda_{a_n})^{s-1}\mathcal H^{s-1}(E_{\theta, T_{a_1\cdots a_n}(x)}).
\]
Plugging this into our definition of $\mu^2_x$ we see that the measures $\mu^2$ and $\mu^1$ coincide whenever they are both defined. Since $\mu^1$ was a disintegration of Hausdorff measure on $E$, we have the following theorem.

\begin{theorem}
Suppose that the conditions of Theorem \ref{FractalThm} are satsified. Then the probability measures $\mu^2_x$, which are the normalised $(s-1)$-dimensional Hausdorff measure on slices through $E$, disintegrate the measure $\mathcal H^s$ on $E$.
\end{theorem}

\section{Further Comments, Examples and Questions}
We begin by demonstrating that the example given in the introduction really is a special case of Theorem \ref{FractalThm}. First we need a strengthening of Marstrand's projection theorem for self similar sets with uniform contraction.

\begin{prop}
Let $E$ be a self similar set Hausdorff dimension $s>2$ for which the generating IFS does not contain rotations and for which each contraction has the same contraction ratio. Then for almost every $\theta=(\theta_1,\theta_2)\in[0,\pi)^2$ the orthogonal projection of s-dimensional Hausdorff measure on $E$ down to the line $l_{\theta}$ is an absolutely continuous measure with continuous density.
\end{prop}
\begin{proof}
This proposition, which is probably classical, is proved by a simple convolution argument analagous to one given by Solomyak in \cite{SolomyakAC} to prove that Bernoulli convolutions associated to a parameter $\beta\in(1,\sqrt{2})$ are absolutely continuous with countinuous density. If the set $E$ is generated by contractions $S_1\cdots S_l$ where $S_i(\underline x)=\lambda_i(\underline x)+\underline a_i$ then we can write the measure $\nu_{\theta}$ as the distribution of the sums
\[
\sum_{n=1}^{\infty}\lambda^i(\pi_{\theta}\underline a_{i_n}).
\]
where the $i_n$ are picked uniformly at random from the set $\{1,\cdots l\}$. But these sums can be decomposed into odd an even terms, so we see that $\nu_{\theta}=\nu_{\theta}^{odd}*\nu_{\theta}^{even}$ where these are the measures which give the distribution of the above sums restricted to odd and even terms respectively. Now the Hausdorff dimension $s>2$ is the unique solution of
\[
\sum_{i=1}^l \lambda^s=1,
\]
and so that we see that if $\lambda$ were to be replaced with $\lambda^2$ then the Hausdorff dimension of the corresponding set would be $\frac{s}{2}>1$. In particular, $\nu_{\theta}^{odd}$ and $\nu_{\theta}^{even}$ both absolutely continuous for almost all $\theta$, since they correspond to projections of Hausdorff measure on sets of dimension $\frac{s}{2}>1$. Hence the convolution $\nu_{\theta}=\nu_{\theta}^{odd}*\nu_{\theta}^{even}$ is almost surely absolutely continuous with continuous density, since the convolution of two absolutely continuous measures is absolutely continuous with continuous density.
\end{proof}

The Menger sponge has Hausdorff dimension $\frac{\log(20)}{\log(3)}>2$, it is a self similar set without rotations and satisfies the slice coding condition, thus we have that projections onto subplanes are absolutely continuous with bounded density with probability $1$, and hence by Theorem \ref{FractalThm} we have that almost every plane slice through them has positive finite $s-1$ dimensional Hausdorff measure.

{\bf Question 1:} In loose terms, the above proposition showed that for self similar sets $E$ with uniform contraction ratios and without rotations one can expect more regularity of the measures $\nu_{\theta}(E)$ when the Hausdorff dimension of $E$ is larger. Does one have such a principle if the condition that the contraction ratios are uniform is removed? What about for general sets without any self-similarity?


{\bf Question 2:} Does a self similar set $E$ for which the generating contractions do not contain rotations automatically satisfy the conditions of Definition \ref{DeltaDefn}?


{\bf Question 3:} Is the statement `$\mu_{n,x}\to\lambda|_{I_{\beta}}$ in the weak-star topology for Lebesgue almost-every $x\in I_{\beta}$' equivalent to the statement `$\nu_{\beta}$ is absolutely continuous'? What about with the measures $\nu_{\beta,x}$? Or for the analagous questions on slices and projections of fractals?


{\bf Question 4} Suppose that $\nu_{\beta}$ is singular. Can one describe the measures $m_x$? Do the quantities $\frac{m_x[0]}{m_x[1]}$ mean anything? When $h_{\beta}$ is well defined they relate in a natural way to $h_{\beta}$ through the formulation of $m^1_x$ 


{\bf Question 5:} Do there exist values of $\beta$ for which $\nu_{\beta}$ is absolutely continuous with unbounded density? We note that Feng and Wang found some non-Pisot values of $\beta$ for which $\nu_{\beta}$ is either singular or is absolutely continuous with unbounded density. One might hope that geometric analytic methods may forbid the possibility that $\mathcal E_{\beta}(x)$ has zero Hausdorff measure for each value of $x$, and hence rule out the possibility that $\nu_{\beta}$ is absolutely continuous with unbounded density, this would be very interesting as it would provide non-Pisot examples of singular Bernoulli convolutions.

\section*{Acknowledgements}
Many thanks to Karma Dajani for some extremely helpful discussions. This work was supported by the Dutch Organisation for Scientific Research (NWO) grant number 613.001.022.
\bibliographystyle{plain}
\bibliography{betaref}

\end{document}